\documentclass{ijuc}
\usepackage[T1]{fontenc}%------------------------------------->% permet d'Ã©crire  directement avec les absents
\usepackage{amssymb,amscd}
\usepackage{latexsym}
\usepackage{amssymb}
\usepackage{amsmath}
\usepackage{amsfonts}
\usepackage{amscd}
\usepackage{graphicx}
\usepackage[all]{xy}
\usepackage{amsthm}
\usepackage[latin1]{inputenc}
\usepackage[francais]{varioref}
\usepackage[frenchb,english]{babel}
\usepackage{newlfont}
\usepackage{multicol}
\usepackage[pdftex=true,hyperindex=true,colorlinks=true]{hyperref}
\begin{document}

\newtheorem{thm}{Theorem}[section]
\newtheorem{lem}[thm]{Lemma}
\newtheorem{prop}[thm]{Proposition}
\newtheorem{cor}[thm]{Corollary}
\newtheorem{ex}[thm]{Example}
\newtheorem{rem}[thm]{Remark}
\newtheorem{notation}[thm]{Notation}
\newtheorem{prob}[thm]{Problem}
\newtheorem{note}[thm]{Notes}
\newtheorem{proposition}{Proposition}[section]
\theoremstyle{plain}
\newtheorem{remark}{Remark}[section]
\theoremstyle{plain}
\newtheorem{theorem}{Theorem}[section]
\theoremstyle{plain}
\newtheorem{corollary}{Corollary}[section]
\theoremstyle{plain}
\newtheorem{interpretation}{Interpretation}[section]
\theoremstyle{plain}
\newtheorem{lemma}{Lemma}[section]
\theoremstyle{plain}
\newtheorem{definition}{Definition}[section]
\theoremstyle{plain}
\newtheorem{example}{Example}[section]
\theoremstyle{plain}
\newtheorem{notations}{Notations}[section]
\theoremstyle{plain}
\newenvironment{prev}[1][\underline{Proof}]{\textbf{#1.} }{\ \rule{0.5em}{0.5em}}
\theoremstyle{plain}
\newtheorem{definitions}{Definitions}[section]
\theoremstyle{plain}
\newtheorem{examples}{Examples}[section]

\newtheorem{thmA}{Theorem}
\renewcommand{\thethmA}{}

\theoremstyle{definition}

\newtheorem{defi}[thm]{Definition}
\renewcommand{\thedefi}{}
%%%%%%%%%%%%%%%%%%%%%%%
\def\new#1{\textbf{#1}}
%%%%%%%%%%%%%%%%%%%%
\input amssym.def

\title{\textbf{Residuated Multilattice as set of  Truth Values for Fuzzy Rough Sets}}

\author{G. Nguepy Dongmo \inst{1}\email{\makeatletter dongmogaelc@gmail.com} \and B. B. KOGUEP NJIONOU\inst{1}\email{\makeatletter koguep@yahoo.com} \and L. Kwuida\inst{2}\email{\makeatletter leonard.kwuida@bfh.ch} \and M. Onabid\inst{1}\email{\makeatletter mathakong@yahoo.fr}}
\institute{Department of Mathematics and Computer Science, University of Dschang, Cameroon
\and Bern University of Applied Sciences, Switzerland}

\maketitle

\begin{abstract}
In 2004 Anna Maria Radzikowska et al \cite{RK2004} investigated the fuzzy rough sets where the set of truth values is an arbitrary residuated lattice. In this paper, we extend their work by considering a residuated multilattice $M$ as the set of truth values. $M$-fuzzy rough sets are defined using the residuation operators provided by  residuated multilattice $M$. Depending on classes of binary fuzzy relations, we define several classes of $M$-fuzzy rough sets and investigate some properties of these classes.
\end{abstract}

\textbf{\emph{Key Words}}: Formal concept analysis, Rough set, Fuzzy set.

\section{Introduction}\label{sec:Intro}
Several mathematical methods have been proposed to deal with uncertain, incomplete and vague information and applications in different fields of mathematics and computer science. 
Rough sets theory (RST) is a mathematical tool for representing and processing information from data tables. It was first described by the Polish mathematician Zdzis\l{}aw I.\,Pawlak~\cite{PZ1982}. On the other hand,  fuzzy set theory (FST) offers techniques to analyze uncertain and imprecise data. Therefore many efforts have been made to combine the RST and FST in order to get structures than can deal with the two theories. Such structures are \emph{fuzzy rough sets} and \emph{rough fuzzy sets} and have been proposed in~\cite{BK2014, BK2015, BK2017, DP1990}. In \cite{RK2004} Anna Maria Radzikowska have extensively investigated fuzzy rough sets taking an arbitrary residuated lattice as underlying set of truth values. 

In this paper, we consider fuzzy rough sets with residuated multilattice as set of truth values. The rest of the paper is organized  as follows: In Section~\ref{sec:basics} we recall basics notions to make this paper self-contained. Section~\ref{sec:FRS on RML} defines the $M$-fuzzy rough operators, the $M$-fuzzy rough sets and gives some elementary properties. In Section~\ref{sec:MFRS classes} we investigate some classes of fuzzy rough sets with respect to reflexivity, symmetry and transitivity. Section~\ref{sec:conclusion} concludes the paper.

\section{Preliminaries}\label{sec:basics}
In this section we introduces several notions from lattice and multilattice theory in order to make our paper self contained. Let $(P,\leq)$ be a poset and $X\subseteq P$. We denote by $U(X)$ (resp. $L(X)$) the set of upper (resp. lower) bounds of $X$.  
The \textbf{supremum} (resp. \textbf{infimum}) of $X$ is the least (resp. greatest) element of $U(X)$ (resp. $L(X)$), whenever it exists. The supremum (resp. infimum) of $X$ is denoted by $\vee X$ or ${\sup}X$ (resp. $\wedge X$ or ${\inf}X$).  A \new{lattice} is a poset $(P,\le)$ in which any pair of elements has a supremum and an infimum.  
If every subset of $P$ has a supremum and an infimum then $(P,\le)$ is called a \new{complete lattice}~\cite{DP2002}.
A subset $X\subseteq P$ is called a \textbf{chain} (resp. \new{antichain}) if for every $x,y\in X$ we have 
$x\leq y$ or $y\leq x$ (resp. $x\not\leq y$ and $y\not\leq x$)\footnote{We write $x{\parallel}y$ to mean that $x\not\leq y$ and $y\not\leq x$.} .  
A poset $(P,\leq)$ is said to be \textbf{coherent} if every chain has a supremum and an infinimum~\cite{MOR2013}.

To extend the notion of lattice, Benado~\cite{BM1955} introduced multilattices. 
A \textbf{multisupremum} (resp. \textbf{multiinfimum}) of $X$ is a minimal (resp. maximal) element of $U(X)$ (resp. $L(X)$). The set of multisuprema (resp. multiinfima) of $X$ is denoted by $\sqcup X$ (resp. $\sqcap X$). 
%m-sup($X$) (resp. m-inf($X$)).\\
For $x,~y\in P$ we simply write %and $X\subseteq M$, we write
$U(x),~L(x),~x\sqcup y,~x\sqcap y$ %, $\sqcup X$, $\sqcap X$ 
for $U(\{x\}),~L(\{x\})$, $\sqcup\{x,y\}$, $\sqcap\{x,y\}$, respectively. We set
\[{\downarrow}a:=\{x\in P: x\leq a\} \text{ and } {\uparrow}a:=\{x\in P: a\leq x\}, \text{ for any } a\in P.\] The upper (resp. lower) closure of $X$ is ${\uparrow} X=\underset{x\in X}{\bigcup} {\uparrow}x$\quad (resp. ${\downarrow}X=\underset{x\in X}{\bigcup} {\downarrow}x$). 
Note that ${\uparrow}x=U(x)$ and $\sqcup X \subseteq U(X) \subseteq {\uparrow}X$. The dual also holds.

 \begin{defi}~\cite{CCGMO2014}
 A poset $(M,~\leq)$ is called \textbf{multilattice} if, for all $a,b,x\in M$ 
 \begin{itemize}
     \item $a, b\leq x\implies \exists z\in a\sqcup b$ such that $z\leq x$, and
     \item $a, b\geq x\implies \exists z\in a\sqcap b$ such that $z\geq x$. 
 \end{itemize}
A \textbf{complete multilattice}~\cite{MOR2007} is a multilattice $(M,\leq)$ in which $\sqcup X$ and $\sqcap X$ are non empty for any $X\subseteq M$.
\end{defi}
 Any lattice $(L,\wedge,\vee)$ is a multilattice since for all $a,b\in L,~a\sqcap b=\{a\wedge b\}$ and $a\sqcup b=\{a\vee b\}$. 
Whenever $\sqcap X$ or $\sqcup X$ is a singleton, it is denoted by $\bigwedge X$ or $\bigvee X$. Any complete lattice is also a complete multilattice. A multilattice will be called \textbf{pure} if it is not a lattice. 
Fig.~1 below shows an example of a complete and pure multilattice. 
 \hspace{-1cm}
$$\xymatrix@-1pc@M= 1.5pt{ &\top&  \\
c\ar@{-}[ur]& &d\ar@{-}[ul] \\
a\ar@{-}[u]\ar@{-}[urr]& &b\ar@{-}[u]\ar@{-}[ull] \\
& \bot\ar@{-}[ur]\ar@{-}[ul]& \\
} %\captionof{figure}{Example of CNFI define on a power set}
$$
\begin{center}
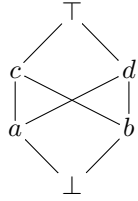
   \textbf{Fig.~1}: The Multilattice (M6,$\leq$)
\end{center}

In order to introduce our truth degree structure, we start with the \new{residuation}:
\begin{defi}~\cite{CCGMO2014}
A \new{pocrim}  (partially ordered commutative residuated integral monoid) is a  structure $(A,\leq,\odot,\rightarrow,\top)$ such that 
\begin{itemize}
\item[(1)] $(A,\odot,\top)$ is a commutative monoid with neutral element $\top$,
\item[(2)] $(A,\leq)$ is a poset with a top element $\top$, and
\item[(3)] $a\odot b \leq c \iff a\leq b\to c$, for all $a, b, c\in A$. \quad (adjointness condition)
\end{itemize}
\end{defi}
\noindent
If $(L, \leq, \bot, \top)$ is a bounded lattice and $(L,\leq,\odot,\to,\top)$ a pocrim then $(L,\leq,\odot,\to,\bot,\top)$ is called a \new{residuated lattice}.  
A residuated lattice is \new{complete} if the underlying poset is a complete lattice. The following properties hold in pocrims:
\begin{prop}~\cite{CCGMO2011}
Let $(A,~\leq,~\odot,~\rightarrow,~\top)$ %$(A,~\leq,~\odot,~\rightarrow,~\bot,~\top)$
be a pocrim and $a, b, c\in A$. Then
\begin{itemize}
\item[\textbf{P1}] $a\odot b\leq a$ and $a\odot b\leq b$;
\item[\textbf{P2}] $a\odot(a\rightarrow b)\ \leq\ a\ \leq\ b\rightarrow (a\odot b)$ and $a\odot(a\rightarrow b)\ \leq\  b\ \leq\ a\rightarrow (a\odot b)$;
\item[\textbf{P3}] If $a\leq b$, then $a\odot c\  \leq\  b \odot c$,\quad $c\rightarrow a\ \leq\  c\rightarrow b$,\quad  and $b\rightarrow c\ \leq\  a\rightarrow c$;        \item[\textbf{P4}] $a\rightarrow (b\rightarrow c)= b\rightarrow (a\rightarrow c)=(a\odot b)\rightarrow c$;
\item[\textbf{P5}] $(a\rightarrow b)\odot (b\rightarrow c)\ \leq\  a\rightarrow c$;
\item[\textbf{P6}] $a\rightarrow b\ \leq\  (a\odot c)\rightarrow( b\odot c)$;
\item[\textbf{P7}] $a\rightarrow b\ \leq\  (c \rightarrow a)\rightarrow(c\rightarrow b)$ and $a\rightarrow b\ \leq\  (b \rightarrow c)\rightarrow(a \rightarrow c)$;
\item[\textbf{P8}] $\top\rightarrow a=a$ and $a\rightarrow \top=\top $;
\item[\textbf{P9}] $a\leq b$ if and only if $a\rightarrow b=\top$.
\end{itemize}
\end{prop}

On any residuated lattice $(L, \wedge, \vee, \odot, \rightarrow, \bot, \top)$ we can define a unary operator $*$ by ${a^*:=a\rightarrow \bot}$, for any $a\in L$. 
\begin{prop}\cite{GO1967,RK2004} 
Let $(L, \wedge, \vee, \odot, \rightarrow, \bot, \top)$ be a residuated lattice. For every $a,b,c\in L$ and for any indexed family $(x_i)_{i\in I}$ of elements of $L$, we have: %\marginpar{replace x with a,\\ y with b and  \\ z with c. \\ You can keep $x_i$}
\begin{itemize}
\item[\textbf{L1}] $a\leq a^{**},~a^*=a^{***},~a^{**}\rightarrow b^{**}=b^*\rightarrow a^*,~(a\odot b)^* = a\rightarrow b^*$;
    \item[\textbf{L2}] $(\underset{i\in I}{\vee} x_i)\rightarrow c =\underset{i\in I}{\wedge}(x_i\rightarrow c)$, $ c\rightarrow (\underset{i\in I}{\wedge x_i})=\underset{i\in I}{\wedge}(c\rightarrow x_i)$;
        \item[\textbf{L3}]$\underset{i\in I}{\vee}(c\odot x_i)=c\odot (\underset{i\in I}{\vee}x_i)$, $c\odot (\underset{i\in I}{\wedge}x_i)\leq \underset{i\in I}{\wedge}(c\odot x_i)$;
        \item[\textbf{L4}] $(\underset{i\in I}{\vee}~x_i)^*=\underset{i\in I}{\wedge}~x_i^*$;
        \item[\textbf{L5}] $(\underset{i\in I}{\wedge}~x_i)^*\geq \underset{i\in I}{\vee}~x_i^*$;
        \item[\textbf{L6}] $(a\odot b)^*=(a\rightarrow b^*)$
\end{itemize}
\end{prop}

\begin{defi}~\cite{CCGMO2014}
A \textbf{residuated multilattice} is a pocrim, whose underlying poset is a multilattice. If in addition, there exists a bottom element, the residuated multilattice is said to be bounded.
A residuated multilattice %$(M,~\leq,~\odot,~\rightarrow,~\bot,~\top)$ 
is complete if the underlying multilattice is complete.
\end{defi}
From now on, $\textbf{M}:=(M,\leq,\odot,\rightarrow,\bot,\top)$ will denote a complete residuated multilattice.  The operations $\odot$, $\rightarrow$ and $*$ can be extended to $\mathcal{P}(M)-\{\emptyset\}$ as follows, for $A, B\in \mathcal\mathcal{P}(M)-\{\emptyset\}$:
\begin{align*}
    A\odot B &:=\{a\odot b : ~a\in A~and~b\in B\},~ \\
    A\rightarrow B& :=\{a\rightarrow b : ~a\in A~and~b\in B\}, \\
    A^*&:=\{a^* : ~a\in A\}.
\end{align*}
In particular for $z\in M$, $A\to z,~A\odot z$ will stand respectively for $A\to \{z\},~A\odot\{z\}$.
\begin{prop}\cite{CCGMO2011} Let $M$ be a complete residuated multilattice and $x, y, z \in M$. Then:
%\marginpar{brackets in M11?, Cor 2.7?}
\begin{itemize}
\item[\textbf{M1}] $x\odot y,~x\odot(x\rightarrow y)\in \downarrow (x\sqcap y)$;
\item[\textbf{M2}] $(x\odot y) \sqcup (x\odot z)\subseteq x\odot (y\sqcup z)$;
\item[\textbf{M3}] $x\odot (y\sqcap z)\subseteq \downarrow[(x\odot y)\sqcap (x\odot z)]$
\item[\textbf{M4}] $x\odot (y\sqcup z)\subseteq \uparrow [(x\odot y)\sqcup (x\odot z)]$
    \item[\textbf{M5}] $(x\sqcap y)\rightarrow z\subseteq \uparrow [(x\rightarrow z)\sqcup (y\rightarrow z)];$
        \item[\textbf{M6}]$(x\sqcup y)\rightarrow z\subseteq\downarrow [(x\rightarrow z)\sqcap(y\rightarrow z)];$
            \item[\textbf{M7}] $(x\rightarrow z)\sqcap (y\rightarrow z)\subseteq (x\sqcup y)\rightarrow z;$
                \item[\textbf{M8}] $z\rightarrow (x\sqcup y)\subseteq \uparrow[(z\rightarrow x)\sqcup (z\rightarrow y)]$;
                    \item[\textbf{M9}] $z\rightarrow (x\sqcap y)\subseteq \downarrow[(z\rightarrow x)\sqcap(z\rightarrow y)]$;
\item[\textbf{M10}] $x\leq x^{**},~x^*=x^{***},~x^{**}\rightarrow y^{**}=y^*\rightarrow x^*,~(x\odot y)^* = x\rightarrow y^*$;
        \item[\textbf{M11}] $(x\odot y)^*= x\rightarrow y^*$
    \item[\textbf{M12}] $(x\sqcap y)^*\subseteq \uparrow (x^*\sqcup y^*)$;
    \item[\textbf{M13}] $(x\sqcup y)^*\subseteq \downarrow (x^*\sqcap y^*)$;
    \item[\textbf{M14}] $(x^*\sqcap y^*)\subseteq (x\sqcup y)^*$.
\end{itemize}
\end{prop}

\begin{cor}\label{sec:cor}
In a complete residuated multilattice $\textbf{M}$, the following conditions hold for all $X\subseteq M$ and $z\in M$:
\begin{itemize}
\item[1.] $\sqcup(z\odot X)\subseteq z\odot (\sqcup X)$;
\item[2.] $\left(\sqcap X\right)^*\ \subseteq\  {\uparrow}\left(\sqcup X^* \right)$;
\item[3.]$\left(\sqcup X\right)^*\ \subseteq\  {\downarrow}\left(\sqcap X^* \right)$;
 \item[4.]$\sqcap (X\rightarrow z)\subseteq \left(\sqcup X \right)\rightarrow z$. In particular, for $z=\bot$ we have: $\sqcap X^*\subseteq \left(\sqcup X \right)^*$;
 \item[5.] $\sqcap (z\rightarrow X)\subseteq z\rightarrow (\sqcap X)$.
 \item[6.] ${\downarrow}\left(\sqcup X\right)^*\ =  {\downarrow}\left(\sqcap X^* \right)$ and ${\uparrow}\left(\sqcap X\right)^*\ \subseteq  {\uparrow}\left(\sqcup X^* \right)$
\end{itemize}
\end{cor}
\begin{proof}
\begin{itemize}
\item[1.] Let $m\in \sqcup (z\odot X)$, let's show that there exists $t\in\sqcup X$ such that $m=z\odot t$.\\
    $m\in \sqcup (z\odot X)$, that is, $m\in\underset{x\in X}{\sqcup}(z\odot x)$ then, for all $x\in X$, $z\odot x\leq m$, by adjointness condition, for all $x\in X$, $x\leq z\rightarrow m$ and then, there exists $t\in \sqcup X$ such that $t\leq z\rightarrow m$ which implies $z\odot t\leq m$. But for all $x\in X$, $x\leq t$ then, for all $x\in X$, $z\odot x\leq z\odot t$, therefore, by the minimality, $m=z\odot t$.
\item[2.] Let $m\in (\sqcap X)^*$. We are looking for an element $t\in (\sqcup X^*)$ such that $t\leq m$.\\
    $m\in (\sqcap X)\rightarrow \bot$ then, there exists $m'\in \sqcap X$ such that $m=m'\rightarrow\bot$. On the other hand, as $m'\in \sqcap X$, we have that, for all $x\in X$, $m'\leq x$ and using $\textbf{P3}$, for all $x\in X$, $x\rightarrow\bot\leq m'\rightarrow \bot$ then, there exists $t\in \underset{x\in X}{\sqcup} \{x\rightarrow\bot\}=\big(\sqcup X^*\big)$ such that $t\leq m'\rightarrow\bot$ that is, $t\leq m$.
    \item[3.] Let $m\in (\sqcup X)^*$. Let's find an element $t\in (\sqcap X)^*$ such that, $m\leq t$.\\
        $m\in (\sqcup X)\rightarrow\bot$ then, there exists $m'\in \sqcup X$ such that, $m=m'\rightarrow\bot$. Since $m'\in \sqcup X$, for all $x\in X$, $x\leq m$ by using $\textbf{P3}$, for all $x\in X$, $m\rightarrow\bot\leq x\rightarrow \bot$ then, there exists $t\in\underset{x\in X}\sqcap\{x\rightarrow \bot\}=\big(\sqcap X^*\big)$ such that, $m'\rightarrow \bot\leq t$ that is, $m\leq t$.
        \item[4.] Let $m\in \sqcap(X\rightarrow z)$. Let'us show that, there exist $m''\in \sqcup X$ such that $m=m''\rightarrow z$.\\
            $m\in \sqcap(X\rightarrow z)$ then, for all $x\in X$, $m\leq x\rightarrow z$, using adjointness, we have for all $x\in X$, $x\odot m\leq z$ and so there exists $m'\in \underset{x\in X}{\sqcup}\{x\odot m\}=\sqcup \big(m\odot X\big)$ such that $m'\leq z$. By the previous (item 1), there exists $m''\in \sqcup X$ such that $m'=m\odot m''\leq z$ which implies $m\leq m''\rightarrow z$. On the other hand, as $m''\in\sqcup X$, for all $x\in X$, $x\leq m''$ and using $\textbf{P3},$ for all $x\in X$, $m''\rightarrow z\leq x\rightarrow z$ then, there exits $m'''\in \sqcap \big(X\rightarrow z\big)$ such that $m''\rightarrow z\leq m'''.$ Therefore, $m\leq m''\rightarrow z\leq m'''$ and, as $m$ and $m'''$ should be either equal or incomparable, we obtain $m=m'''=m''\rightarrow z$.
            \item[5.] Inclusion 5 follows the same pattern as item $4$.
    \item[6.] The equality ${\downarrow}\left(\sqcup X\right)^*\ =  {\downarrow}\left(\sqcap X^* \right)$ follows from the inclusions below:
\begin{align*} 
    \left(\sqcup X\right)^*\ \subseteq\  {\downarrow}\left(\sqcap X^* \right) &\implies {\downarrow}\left(\sqcup X\right)^*\ \subseteq\  {\downarrow}\left(\sqcap X^* \right) \qquad \text{by 3.}\\
    \sqcap X^*\subseteq \left(\sqcup X \right)^* &\implies {\downarrow}\left(\sqcap X^*\right)\ \subseteq\  {\downarrow}\left(\sqcup X \right)^* \qquad \text{by 4.}
\end{align*}
For the inclusion ${\uparrow}\left(\sqcap X\right)^*\ \subseteq\   {\uparrow}\left(\sqcup X^* \right)$, note that 
from 2. we get $\left(\sqcap X\right)^*\ \subseteq\  {\uparrow}\left(\sqcup X^* \right)$, then  ${\uparrow}\left(\sqcap X\right)^*\ \subseteq\  {\uparrow}\left(\sqcup X^* \right)$. 
\end{itemize}
\end{proof}
\begin{defi}~\cite{MOR2013}
Let $(M_1,\leq_1)$, $(M_2,\leq_2)$ be two multilattices and $(P,\leq)$ be a poset, and $\blacklozenge:~M_1\times M_2\rightarrow P$ be a mapping between them. We say that $\blacklozenge$ is:
\begin{itemize}
     \item[(i)] soft left-continuous in the first argument if for every non empty subset $K_1\subseteq M_1$ and elements $m_2\in M_2$ and $p\in P$ such that $k\blacklozenge m_2\leq p$ for every $k\in K_1$, then there exists $m_1\in \sqcup K_1$ satisfying $m_1\blacklozenge m_2\leq p$.
     \item[(ii)] soft left-continuous in the second argument if for every non empty subset $K_2\subseteq M_2$ and element $m_1\in M_1$ and $p\in P$ such that $m_1\blacklozenge k\leq p$ for every $k\in K_2$, then there exists $m_2\in \sqcup K_2$ satisfying $m_1\blacklozenge m_2\leq p$.
     \item[(iii)] soft left-continuous if it is soft left-continuous in the both argument.
\end{itemize}
\end{defi}
If we use a short hand notation $K_1 \blacklozenge m_2\leq p$ to mean that $k\blacklozenge m_2\leq p$ for every $k\in K_1$, then the soft left-continuous in the first argument can be neatly expressed by:
\[\forall K_1\in \mathcal{P}(M_1)\setminus\{\emptyset\}, \quad   K_1 \blacklozenge m_2\leq p \implies \exists m_1\in \sqcup K_1,\ m_1 \blacklozenge m_2 \le p.\]
\begin{prop}\label{sec: prop}
The operator $\odot$ defined on $\textbf{M}$ is soft left continuous.
\end{prop}
\begin{proof}
Since $\odot$ is commutative, we will only show that it is soft left continuous for the first argument.
Let $\emptyset\neq K\subseteq M,~y\in M$ and $z\in M$. 
We assume that $k\odot y\leq z$ for all $k\in K$.  We are looking for an element $x\in \sqcup K$ such that $x\odot y\leq z$.
For all $k$ in $K$, $k\leq y\rightarrow z$.  Therefore $y\rightarrow z$ is an upper bound of $K$. Since $M$ is a complete multilattice, there is an element $x\in\sqcup K$ such that $x\leq y\rightarrow z$, i.e. $x\odot y \le z$.
\end{proof}
%\begin{prop}\label{sec:prop1}
%For every $\emptyset\neq K\subseteq M$ and $y, z\in M$.
% If $z\leq y\rightarrow k$ for all $k\in K$,  then there exists $x\in\sqcap K$ satisfying $z\leq %y\rightarrow x$.
%\end{prop}
%\begin{proof}
%Suppose $z\leq y\rightarrow k$, for all $k\in K$. We have $z\odot y\leq k$ for all $k\in K$. Therefore, %$z\odot y$ is a lower bound of $K$.
%Hence, as $M$ is a complete multilattice, there is $x\in\sqcap K$ such that $z\odot y\leq x$ and applying %the adjoint property again, we have that there exists $x\in\sqcap K$ such that $z\leq y\rightarrow x$.
%\end{proof}
\section{Fuzzy Rough Sets based on Residuated Multilattice}~\label{sec:FRS on RML}
Fuzzy rough set theory (FRST) was introduced in \cite{DP1990} by Dubois and Prade. In this section we explore how to use residuated multilattices as underlying set of truth values in fuzzy rough set theory. 
We start by recalling some basic notions on $L$-fuzzy rough set as defined in \cite{RK2004}, and then extend these notions to the context of $M$-fuzzy rough set, where $M$ is a complete residuated multilattice.

\begin{defi} \cite{GO1967}\label{def:fuzzy set and fuzzy relation} %\marginpar{check the ref}
Let $(L, \wedge, \vee, \odot, \rightarrow, \bot, \top)$ be a complete residuated lattice and let $X$ be a non empty universe. Any mapping $R:X^n\longrightarrow L$ for $n\geq 2$, is called an \new{$n$-ary $L$-fuzzy relation} on $X$. For $n=2$, $R$ is called a binary $L$-fuzzy relation on $X$. An \new{$L$-fuzzy set} in $X$ is just a map $f:X\longrightarrow L$. The family of all $L$-fuzzy sets in $X$ will be denoted by $L^X$. 
\end{defi}

Some properties of a binary relation are given below:
\begin{defi} \cite{GO1967, RK2004} 
Let $R$ be a binary $L$-fuzzy relation, we say that $R$ is:
\begin{itemize}
    \item[$\bullet$] \new{reflexive} if $R(x,x)=\top$, for every $x\in  X$;
    \item[$\bullet$] \new{symmetric} if $R(x,y)=R(y,x)$, for all $x,y\in X$;
    \item[$\bullet$] \new{euclidean} if $R(z,x)\odot R(z,y)\leq R(x,y)$, for all $x,y,z\in X$;
    \item[$\bullet$] \new{transitive} if $R(x,z)\odot R(z,y)\leq R(x,y)$, for all $x,y,z\in X$.
\end{itemize}
\end{defi}
Let $(L,\leq, \odot, \rightarrow, \bot, \top)$ be a complete residuated lattice and $X$ a non-empty universe. The  order relation $\le$ on $L$ is extended on $L^X$ by $f\leq g :\iff f(x)\leq g(x)$ for all $x\in X$. The operations $\wedge$, $\vee$, $\odot$, $\rightarrow$, $^*$, $\Uparrow$ and $\Downarrow$ are defined on $L^X$ as follows: 
\begin{align*}
(f\wedge g)(x)&:=f(x)\wedge g(x), & (f\vee g)(x)&:=f(x)\vee g(x),\\
(f\odot g)(x)&:=f(x)\odot g(x) & (f\rightarrow g)(x)&:=f(x)\rightarrow g(x) \\
\Uparrow(x)&:=\top \text{ and } \Downarrow(x):=\bot & f^*(x)&:=(f(x))^*  
\end{align*}
%following are some operations on $L$-fuzzy set: for all $f,g\in L^X$ and for every $x\in X$, we have:
% \begin{defi} \cite{RK2004} \marginpar{check the ref}
% \begin{itemize}
% \item[$\bullet$] $f\leq g \Leftrightarrow f(x)\leq g(x)$;
% \item[$\bullet$] $f^*(x)=(f(x))^*, (f\rightarrow g)(x)=f(x)\rightarrow g(x);~~(f\odot g)(x)=f(x)\odot g(x)$;
%     \item[$\bullet$] $(f\wedge g)(x)=f(x)\wedge g(x)$ and $(f\vee g)(x)=f(x)\vee g(x)$.
% \item[$\bullet$] We define the two maps $\Uparrow$ and $\Downarrow$ from $X$ to $L$ respectively by\\
%     $\Uparrow(x)=\top$ and $\Downarrow(x)=\bot$ for all $x\in X$.

% \end{itemize}
\begin{defi} \label{def: approx space} \cite{RK2004}  %\marginpar{check the ref}
An \new{$L$-fuzzy approximation space} is a pair $(X,R)$, where $X$ is a universe and $R$ a binary $L$-fuzzy relation. Two mappings $\underline{R}, ~\overline{R}:L^X\rightarrow L^X$ are defined on $L^X$ by: for every $f\in L^X$ and every $x\in X$,
\begin{align*}
    \underline{R}(f)(x):=\underset{y\in X}{\bigwedge}(R(x,y)\rightarrow f(y)) \quad
    \text{ and }\quad
\overline{R}(f)(x):=\underset{y\in X}{\bigvee}(R(x,y)\odot f(y)).
\end{align*}
$\underline{R}(f)$ (resp. $\overline{R}(f)$) is called a lower (resp. upper)  \new{$L$-fuzzy rough approximation} of $f$. 
\end{defi}
% \begin{center}
% $\displaystyle\underline{R}(f)(x)=\underset{y\in X}{\bigwedge}(R(x,y)\rightarrow f(y))$ (lower $L$-fuzzy rough approximator)\\
% $\displaystyle\overline{R}(f)(x)=\underset{y\in X}{\bigvee}(R(x,y)\odot f(y))$ (upper $L$-fuzzy rough approximator)
% \end{center}
%$\underline{R}(f)$ (resp. $\overline{R}(f)$) is called a lower (resp. upper) $L$-fuzzy rough approximation of $A$

\begin{prop} \cite{RK2004}
Let $(L, \wedge, \vee, \odot, \rightarrow, \bot, \top)$ be a complete residuated lattice. For every $L$-fuzzy approximation space $(X, R)$ and every $f\in L^X$ the following conditions hold:
\begin{itemize}
\item[(i)] $\overline{R}(\Downarrow)=\Downarrow$ and $\underline{R}(\Uparrow)=\Uparrow$;
    \item[(ii)] If $f\leq g$, then $\underline{R}(f)\leq\underline{R}(g)$ and $\overline{R}(f)\leq \overline{R}(g)$;
    \item[(iii)] $\underline{R}(f)\leq (\overline{R}( f^*))^*$,\quad
            $\overline{R}(f)\leq (\underline{R}( f^*))^*$\quad \mbox{ and }\quad
            $(\overline{R}(f))^*=\underline{R}(f^*)$.
\end{itemize}
\end{prop}
%\section{M-Fuzzy Rough Set}

%\subsection{M-Fuzzy sets and M-fuzzy realtions}
Now we can introduce the notion of $M$-fuzzy approximation. 
%As an extension of the definition of $L$-fuzzy set and $L$-fuzzy relation we have the following definitions,
\begin{defi}
Let \textbf{M} be a residuated multilattice. A \textbf{M-fuzzy set} $f$ on a nonempty universe $X$ is a mapping $f:~ X\longrightarrow M$.
\end{defi}
The family of all $M$-fuzzy sets in $X$ will be denoted by $M^X$, the operators $\odot, \empty^*, \rightarrow$\\ and the maps $\Uparrow, \Downarrow$ are defined exactly as in $L^X$, replacing the residuated lattice $L$ with a residuated multilattice $M$. Similarly the notions of \new{$M$-fuzzy relation}, reflexivity, symmetry, transitivity, \new{$M$-tolerance} and \new{$M$-equivalence} are defined. 

%Repetitions!!!!: commenté?
% \begin{defi}
% Let $(M,~\leq,~\odot,~\rightarrow,~\bot~\top)$ be a residuated multilattice and let $X$ be a non empty universe. Any mapping $R:~X^n \rightarrow M$, for $n\geq 2$, is called a $M$-fuzzy relation. For $n=2$, $R$ is called a binary $M$-fuzzy relation.
% \end{defi}
% The family of all binary $M$-fuzzy relations on $X$ will be denoted by $M^{X^2}$.\\
% Let $R\in M^{X^2}.$ We say that $R$ is: 
% \begin{itemize}
%     \item[$\bullet$] reflexive if $R(x,x)=\top$, for every $x\in  X$
%     \item[$\bullet$] symmetric if $R(x,y)=R(y,x)$, for all $x,y\in X$
%     \item[$\bullet$] euclidean if $R(z,x)\odot R(z,y)\leq R(x,y)$, for all $x,y,z\in X$
%         \item[$\bullet$] Transitive if $R(x,z)\odot R(z,y)\leq R(x,y)$, for all $x,y,z\in X$
% \end{itemize}
% A binary $M$-fuzzy relation $R$ is called an $M$-tolerance relation if $R$ is reflexive and symmetric; if an $M$-tolerance relation is also transitive then it is called an $M$-equivalence relation.\\

%\noindent
For $R\in M^{X^2}$ and $x\in X$, we write $R_x$ to denote the $M$-fuzzy set on $X$ defined by $R_x(y):=R(x,y)$, for every $y\in X$. We  can now  define the $M$-fuzzy rough approximators.
\begin{defi}
Let \textbf{M} be a residuated multilattice. A \new{$M$-fuzzy approximation space} is a pair $(X, R)$ where $X$  is a non empty universe and $R\in M^{X^2}$.
\end{defi}
To define the $M$-fuzzy rough approximators we need certain infima and suprema.
\begin{thm}
Let $(M, \leq, \rightarrow, \bot, \top)$ be a complete residuated multilattice, $f\in M^{X}$ and $R\in M^{X^2}$. Then, for any $x\in X$ the following holds:
\begin{itemize}
    \item[(i)] The set $\{R(x,y)\rightarrow f(y)\mid~y\in X\}$ has an infimum.
    \item[(ii)] The set $\{R(x,y)\odot f(y)\mid~y\in X\}$ has a supremum. 
\end{itemize}
% \begin{center}
% $inf\{R(x,y)\rightarrow f(y)\mid~y\in X\}$ and $sup\{R(x,y)\odot f(y)\mid~y\in X\}$
% \end{center}
\end{thm}
\begin{proof} Because we are working with a complete residuated multilattice, for all $x\in X$ and $f\in M^X$ the sets $\sqcap\{R(x,y)\rightarrow f(y)\mid~y\in X\}$ and $\sqcup\{R(x,y)\odot f(y)\mid~y\in X\}$ are nonempty. We will prove that $\sqcap\{R(x,y)\rightarrow f(y)\mid~y\in X\}$ and $\sqcup\{R(x,y)\odot f(y)\mid~y\in X\}$ are singletons.
\begin{itemize}
\item[(i)] 
Let $x_1$ and $x_2$ be in  $\sqcap\{R(x,y)\rightarrow f(y)~\mid~y\in X\}$. We have $x_1\leq R(x,y)\rightarrow f(y)$ and $x_2\leq R(x,y)\rightarrow f(y)$, for every $y\in X$. Then $x_1\odot R(x,y)\leq f(y)$ and $x_2\odot R(x,y)\leq f(y)$, for every $y\in X$.
Since $\odot$ is  soft left continuous , there exists $a\in\sqcup\{x_1; x_2\}$, such that $a\odot R(x,y)\leq f(y)$. Hence $a\leq R(x,y)\rightarrow f(y)$, for every $y\in X$. 
Hence $a$ is a lower bound of the set $\{R(x,y)\rightarrow f(y)~\mid~y\in X\}$; as $x_1$ and $x_2$ are maximal lower bounds, we obtain that $a=x_1=x_2$. Thus, all multiinfima collapse in one and, so, there is an infimum.
\item[(ii)] Let $x_1$ and $x_2$ $\in\sqcup\{R(x,y)\odot f(y)~\mid~y\in X\}$, we have that $R(x,y)\odot f(y)\leq x_1$ and $R(x,y)\odot f(y)\leq x_2$, for every $y\in X$. Then there exist $b\in \sqcap\{x_1, x_2\}$ such that for every $y\in X$, $R(x,y)\odot f(y)\leq b$. Hence $b$ is an upper bound of $\{R(x,y)\odot f(y)~\mid~y\in X\}$ and $b\leq x_1, x_2$; as $x_1$ and $x_2$ are minimal upper bounds, we obtain $x_1=b=x_2$. Thus, all multisuprema collapse in one, then there is a suprema.
\end{itemize}
%Proof of (i) is fine. But in (ii) you need to show that b is not a b(y).
\end{proof}
\begin{defi}
Let $(M,~\leq,~\odot,~\rightarrow,~\bot,~\top)$ be a complete residuated multilattice and $(X, R)$ be a $M$-fuzzy approximation space. Define the mappings $\underline{R},~\overline{R}:M^X\rightarrow M^X$ by: 
\begin{align*}
    \underline{R}(f)(x):=\bigwedge_{y\in X}(R(x,y)\rightarrow f(y))\quad \text{ and }\quad \overline{R}(f)(x):=\bigvee_{y\in X}(R(x,y)\odot f(y)),
\end{align*}
 for every $f\in M^X$ and every $x\in X$. The operators $\underline{R}$ and $\overline{R}$ are respectively called lower and upper \new{$M$-fuzzy rough approximators}, and  
% \begin{center}
% $\displaystyle\underline{R}(f)(x)=\bigwedge_{y\in X}(R(x,y)\rightarrow f(y))$
% \quad\mbox{ and }\quad
% $\displaystyle\overline{R}(f)(x)=\bigvee_{y\in X}(R(x,y)\odot f(y))$
% \end{center}
% , called a lower and upper $M$-fuzzy rough approximators, respectively, as follows: for every $f\in M^X$ and every $x\in X$,
%
$\underline{R}(f)$ (resp. $\overline{R}(f)$) is called a lower (resp. upper) \new{$M$-fuzzy rough approximation} of $f$.
\end{defi}
\begin{defi}
Let $(X,~R)$ a  $M$-fuzzy approximation space. A pair ($\overline{R}(f), \underline{R}(f))\in M^X\times M^X$, for some $f\in M^X$ is called a $M$-fuzzy rough set in $(X, R)$.
\end{defi}
%\section{Basic Properties of M-Fuzzy Rough Sets}
We are going now to study some properties of $M$-fuzzy rough approximators.
\begin{prop}\label{sec: approxtop}
Let $(M,~\leq,~\odot,~\rightarrow,~\bot,~\top)$ be a complete residuated multilattice. For every $M$-fuzzy approximation space $(X, R)$, every $f\in M^X$ the following conditions hold:
\begin{itemize}
\item[(i)] $\overline{R}(\Downarrow)=\Downarrow$ \quad and \quad $\underline{R}(\Uparrow)=\Uparrow$;
    \item[(ii)] If $f\leq g$,\quad then \quad $\underline{R}(f)\leq\underline{R}(g)$\quad and \quad $\overline{R}(f)\leq \overline{R}(g)$;
        \item[(iii)] $\underline{R}(f)\leq \Big(\overline{R}\big( f^*\big)\Big)^*$,\quad 
            $\overline{R}(f)\leq \Big(\underline{R}\big( f^*\big)\Big)^*$\quad and \quad 
            $\Big(\overline{R}(f)\Big)^*=\underline{R}(f^*)$.

\end{itemize}
\end{prop}
\begin{proof}
\begin{itemize}
\item[(i)] Let $x\in X$;
\begin{align*}
\overline{R}(\Downarrow)(x)=\bigvee_{y\in X}\Big(R(x,y)\odot \Downarrow(y)\Big)
                           =\bigvee_{y\in X}\Big(R(x,y)\odot \bot\Big)
                           =\bigvee_{y\in X}\bot
                           =\bot.
\end{align*}
%For all $x\in X$;
\begin{align*}
\underline{R}(\Uparrow)(x)=\bigwedge_{y\in X}\Big(R(x,y)\rightarrow \Uparrow(y)\Big)
                          =\bigwedge_{y\in X}\Big(R(x,y)\rightarrow \top\Big)
                          =\bigwedge_{y\in X}\top
                          =\top.
\end{align*}
\item[(ii)] Let $x\in X$ and $f\le g$. Then $f(y)\le g(y)$ for all $y\in X$. Thus 
\begin{align*}
f\leq g& \implies R(x,y)\rightarrow f(y)\leq R(x,y)\rightarrow g(y) \text{ for all } y\in X\\
 &\implies \bigwedge_{y\in X} \Big(R(x,y)\rightarrow f(y)\Big)\leq\bigwedge_{y\in X} \Big(R(x,y)\rightarrow g(y)\Big)\\
  &\implies\underline{R}(f)(x)\leq\underline{R}(g)(x).
\end{align*}
and 
\begin{align*}
f\leq g &\implies R(x,y)\odot f(y)\leq R(x,y)\odot g(y)\text{ for all } y\in X\\
&\implies \bigvee_{y\in X} \Big(R(x,y)\odot f(y)\Big)\leq\bigvee_{y\in X} \Big(R(x,y)\odot g(y)\Big)\\
&\implies \overline{R}(f)(x)\leq\overline{R}(g)(x).
\end{align*}
Thus, $f\leq g$  implies $\underline{R}(f)\leq\underline{R}(g)$ and $\overline{R}(f)\leq\overline{R}(g)$. 
%For all $x\in X$;
%Thus, $f\leq g ~~\Rightarrow~~\overline{R}(f)\leq\overline{R}(g)$.\\
\item[(iii)] Let $x\in X$ and $f\in M^X$.
\begin{align*}
\underline{R}(f^*)(x) &=\bigwedge_{y\in X}\Big(R(x,y)\rightarrow f^*(y)\Big) = 
\bigwedge_{y\in X}\Big(R(x,y)\rightarrow \big(f(y)\big)^*\Big)\\
                         &=\bigwedge_{y\in X}\Big(R(x,y)\rightarrow \big(f(y)\rightarrow\bot\big)\Big) =
\bigwedge_{y\in X}\Big(\big(R(x,y)\odot f(y)\big)\rightarrow\bot\Big), \text{ by (P4)}\\
                         &=\bigwedge_{y\in X}\big(R(x,y)\odot f(y)\big)^*.
\end{align*}
Thus 
\begin{align*}
    \left( \underline{R}(f^*)(x)  \right)^* \in \left(\underset{y\in X}{\sqcap}\Big(R(x,y)\odot f(y)\Big)^*\right)^* 
                         &\subseteq\ {\uparrow}\left(\underset{y\in X}{\sqcup}\Big(R(x,y)\odot f(y)\Big)^{**}\right), \text{ by Corollary \ref{sec:cor} (item 2)}.
\end{align*}
% \begin{align*}
% (\underline{R}(f^*)(x))^*&=\left(\bigwedge_{y\in X}\Big(R(x,y)\rightarrow \big(f(y)\big)^*\Big)\right)^*\\
%                          &=\left(\bigwedge_{y\in X}\Big(R(x,y)\rightarrow \big(f(y)\rightarrow\bot\big)\Big)\right)^*\\
%                          &=\left(\bigwedge_{y\in X}\Big(\big(R(x,y)\odot f(y)\big)\rightarrow\bot\Big)\right)^*, ~by~( P4)\\
%                          &=\left(\bigwedge_{y\in X}\big(R(x,y)\odot f(y)\big)^*\right)^*\\
%                          &\in\left(\underset{y\in X}{\sqcap}\Big(R(x,y)\odot f(y)\Big)^*\right)^*\\
%                          &\subseteq\uparrow\left(\underset{y\in X}{\sqcup}\Big(R(x,y)\odot f(y)\Big)^{**}\right),~by~ corollary~2.13 ~(item~2)
% \end{align*}
Then, there is $t\in\underset{y\in X}{\sqcup}\Big(R(x,y)\odot f(y)\Big)^{**}$ such that $t\leq \big(\underline{R}(f^*)(x)\big)^*$. Since $t$ is in $\underset{y\in X}{\sqcup}\Big(R(x,y)\odot f(y)\Big)^{**} $, it follows that $\big(R(x,y)\odot f(y)\big)^{**}\leq t$,  for all $y\in X$. Therefore, $R(x,y)\odot f(y)\leq\Big(R(x,y)\odot f(y)\Big)^{**}\leq t\leq\Big(\underline{R}(f^*)(x)\Big)^*$, for all $y\in X$. 
 Thus, $\displaystyle \overline{R}(f)(x)=\bigvee_{y\in X}\Big(R(x,y)\odot f(y)\Big)\leq  \Big(\underline{R}(f^*)(x)\Big)^*$
for all $x\in X$.
\begin{align*}
(\overline{R}(f^*)(x))^*&=\left( \underset{y\in X}{\sqcup}\Big(R(x,y)\odot\big(f(y)\big)^*\Big)\right)^*\\
                         &\supseteq\left(\underset{y\in X}{\sqcap}\Big(R(x,y)\odot \big(f(y)\big)^*\Big)^{*}\right),\text{ by Corollary \ref{sec:cor} item 4} \\
                         &\stackrel{ \text{M11} }{=}\left(\underset{y\in X}{\sqcap}\Big(R(x,y)\rightarrow \big(f(y)\big)^{**}\Big)\right)
                         \end{align*}
Since $ \underset{y\in X}{\sqcap}\Big(R(x,y)\rightarrow (f(y))^{**}\Big)=\underline{R}(f^{**})(x)$ is a singleton, we therefore have,
\begin{align*}
\Big(\overline{R}(f^*)(x)\Big)^* &=\left(\underset{y\in X}{\sqcap}\Big(R(x,y)\rightarrow \big(f(y)\big)^{**}\Big)\right) 
                         \stackrel{ \text{P3} }{\geq}\underset{y\in X}{\sqcap }\Big(R(x,y)\rightarrow f(y)\Big) 
                         =\underline{R}(f)(x).
\end{align*}
For all $x\in X$
\begin{align*}
\Big(\overline{R}(f)(x)\Big)^*&=\left(\bigvee_{y\in X}\Big(R(x,y)\odot(f(y))\Big)\right)^*\\
                     &\supseteq\underset{y\in X}{\sqcap}\Big(R(x,y)\odot f(y)\Big)^*,\text{ by Corollary \ref{sec:cor} (item 4)} \\
                     &\stackrel{M11}{=}\underset{y\in Y}{\sqcap }\Big(R(x,y)\rightarrow \big(f(y)\big)^*\Big) \\
                     \end{align*}
Since $\underset{y\in X}{\sqcap} \Big(R(x,y)\rightarrow \big(f(y)\big)^*\Big)=\underline{R}(f^*)(x)$ is a singleton, therefore,\\ $\Big(\overline{R}(f)(x)\Big)^*=\underset{y\in X}{\sqcap} \Big(R(x,y)\rightarrow \big(f(y)\big)^*\Big)=\underline{R}(f^*)(x)$.
\end{itemize}
\end{proof}

\section{Selected Classes of M-Fuzzy Rough Set}\label{sec:MFRS classes}
In this section, according to the properties of the $M$-fuzzy relation $R$, we will present some properties of the corresponding $M$-fuzzy rough set, where $M$ denotes the complete residuated multilattice $(M,\leq,\odot,\rightarrow, \bot, \top)$.\\
The following proposition is the continuation of proposition \ref{sec: approxtop}.
\begin{prop}\label{sec: approxbot}
For any reflexive $M$-fuzzy approximation space $(X,R)$ the following properties hold.
\begin{align*}
    \textup{(i)}\quad \underline{R}(\Downarrow) = \Downarrow &&
\textup{(ii)} \quad \overline{R}(\Uparrow) = \Uparrow.
\end{align*}
% \begin{itemize}
% \item[(i)] $\underline{R}(\Downarrow)=\Downarrow$;
% \item[(ii)] $\overline{R}(\Uparrow)=\Uparrow$.
% \end{itemize}
\end{prop}
\begin{proof}
\begin{itemize}
\item[(i)] For every $x\in X$,\\
$\displaystyle\underline{R}(\Downarrow)(x)=\bigwedge_{y\in X}\Big(R(x,y)\rightarrow\bot\Big)\leq R(x,x)\rightarrow \bot= \top\rightarrow\bot=\bot$.
\item[(ii)] For every $x\in X$;\\
$\displaystyle\overline{R}(\Uparrow)(x)=\bigvee_{y\in X}\Big(R(x,y)\odot \top\Big)\geq R(x,x)\odot\top=R(x,x)=\top.$
\end{itemize}
\end{proof}
\begin{prop}
For any $M$-fuzzy approximation space $(X,R)$ the following conditions are equivalent:
\begin{itemize}
\item[(i)] $(X,R)$ is a reflexive fuzzy approximation space;
\item[(ii)] $\underline{R}(f)\leq f$, for every $f\in M^X$;
    \item[(iii)] $f\leq \overline{R}(f)$, for every $f\in M^X$.
\end{itemize}
\end{prop}
\begin{proof}
We are going two show that $(i)\Leftrightarrow(ii)$ and $(i)\Leftrightarrow(iii)$.
\begin{itemize}
\item[$(i)\Rightarrow (ii)$] For every $f\in M^X$ and for every $x\in X$,\\
$\displaystyle\underline{R}(f)(x)=\bigwedge_{y\in X}(R(x,y)\rightarrow f(y))\leq R(x,x)\rightarrow f(x)=\top\rightarrow f(x)\stackrel{\text{P8}}{=}f(x)$. Thus, $\underline{R}(f)\leq f$.\\
\item[$(ii)\Rightarrow (i)$ ]Assume that $R$ is not reflexive. It means that $R(x_0, x_0)\lvertneqq \top$, for some $x_0\in X$. Consider $f=R_{x_0}$. Then we have\\
$\displaystyle\underline{R}(f)(x_0)=\bigwedge_{y\in X}(R(x_0,y)\rightarrow f(y))=\bigwedge_{y\in X}(R(x_0,y)\rightarrow R(x_0,y))\stackrel{P9}{=}\top$. However, $f(x_0)=R(x_0,x_0)\lvertneqq \top=\underline{R}(f)(x_0)$, then $\underline{R}(f)\not\leq f$.\\
\item[$(i)\Rightarrow (iii)$] For every $f\in M^X$ and for every $x\in X$ we have,\\
$\displaystyle\overline{R}(f)(x)=\bigvee_{y\in X}\Big(R(x,y)\odot f(y)\Big)\geq R(x,x)\odot f(x)=\top\odot f(x)=f(x)$. Thus, $f\leq \overline{R}(f)$.\\
\item[$(iii)\Rightarrow (i)$] Assume that $R$ is not reflexive. This means that $R(x_0,x_0)\lvertneqq\top$, for some $x_0\in X$. Let $f$ be defined by $f(x)=\top,$ for $x=x_0$ and $f(x)=\bot$ otherwise. We have\\
$\displaystyle\overline{R}(f)(x_0)=\bigvee_{y\in X}\Big(R(x_0,y)\odot f(y)\Big)= R(x_0,x_0)\odot\top=R(x_0,x_0)\lvertneqq \top.$ Thus $f\not\leq \overline{R}(f)$.
\end{itemize}
\end{proof}
%\subsection{M-Fuzzy Rough Set in symmetric fuzzy approximation space}
Let us now consider, a symmetric $M$-fuzzy approximation space.
\begin{prop}
For any $M$-fuzzy approximation space $(X,R)$ the following conditions are equivalent:
\begin{itemize}
\item[(i)] $(X,R)$ is a symmetric fuzzy approximation space;
\item[(ii)] $\overline{R}(\underline{R}(f))\leq f$, for every $f\in M^X$;
\item[(iii)] $f\leq\underline{R}(\overline{R}(f))$, for every $f\in M^X$.
\end{itemize}
\end{prop}
\begin{proof}
We are going two show that $(i)\Leftrightarrow(ii)$ and $(i)\Leftrightarrow(iii)$.
\begin{itemize}
\item[$(i)\Rightarrow (ii)$ ]For every $f\in M^X$ and for $x\in X$,\\
\begin{align*}
\overline{R}(\underline{R}(f))(x)&=\bigvee_{y\in X}\Big(R(x,y)\odot\bigwedge_{z\in X}\big(R(y,z)\rightarrow f(z)\big)\Big)\\
                                &\stackrel{P3}{\leq}\bigvee_{y\in Y}\Big(R(x,y)\odot\big(R(y,x)\rightarrow f(x)\big)\Big)\\
                                &=\bigvee_{y\in X}\Big(R(x,y)\odot\big(R(x,y)\rightarrow f(x)\big)\Big), \text{ by symmetry of R}\\
                                &\stackrel{P2}{\leq}f(x).
\end{align*}
Thus, $\overline{R}(\underline{R}(f))\leq f$.\\
\item[$(ii)\Rightarrow (i)$] Assume that $R$ is not symmetric, that is there exist $x_0, y_0\in X$ such that $R(x_0, y_0)\neq R(y_0,x_0)$. Consider the following two cases:\\
%\marginpar{wlog!!!}
\begin{itemize}
\item[\textbf{Case 1:}] Assume $R(x_0,y_0)\leq R(y_0, x_0)$. Since $R(x_0, y_0)\neq R(y_0,x_0)$, we have $R(y_0, x_0)\not\leq R(x_0, y_0)$ then by \textbf{P9}, $R(y_0,x_0)\rightarrow R(x_0,y_0)\lvertneqq\top$.\\
For $f=R_{x_0}$, we have,
\begin{align*}
\Big(\overline{R}\big(\underline{R}(f)\big)(y_0)\Big)\rightarrow f(y_0)&=\Bigg(\bigvee_{x\in X}\bigg(R(y_0,x)\odot\bigwedge_{z\in X}\big(R(x,z)\rightarrow R(x_0,z)\big)\bigg)\Bigg)\rightarrow R(x_0,y_0)\\
&\stackrel{P3}{\leq}\Big(R(y_0,x_0)\odot\bigwedge_{z\in X}\big(R(x_0,z)\rightarrow R(x_0,z)\big)\Big)\rightarrow R(x_0,y_0)\\
&=\Big(R(y_0,x_0)\odot\top\Big)\rightarrow R(x_0, y_0)\\
&\lvertneqq \top
\end{align*}
By $\textbf{P9}$, $\overline{R}\big(\underline{R}(f)\big)(y_0)\not\leq f(y_0)$. Thus, $\overline{R}(\underline{R}(f))\not\leq f$.\\
\item[\textbf{Case 2:}] Let $R(x_0,y_0)\not\leq R(y_0, x_0).$ Then, $R(x_0,y_0)\rightarrow R(y_0, x_0)\lvertneqq \top$, by $\textbf{P9}$. Take $f=R_{y_0}$. Similarly we have,
%\marginpar{$f=R_{x_0}$???}
\begin{align*}
\Big(\overline{R}\big(\underline{R}(f)\big)(x_0)\Big)\rightarrow f(x_0)
&=\Bigg(\bigvee_{y\in Y}\Big(R(x_0,y)\odot\bigwedge_{z\in X}\big(R(y,z)\rightarrow R(y_0,z)\big)\Big)\Bigg)\rightarrow R(y_0,x_0)\\
&\leq\Bigg(\bigg(R(x_0,y_0)\odot\bigwedge_{y\in Y}\big(R(y_0,z)\rightarrow R(y_0,z)\big)\bigg)\Bigg)\rightarrow R(y_0,x_0)\\
&\leq R(x_0,y_0)\rightarrow R(y_0,x_0)\lvertneqq \top.
\end{align*}
 Therefore, $\overline{R}(\underline{R}(f))(x_0)\not\leq f(x_0)$. Again $\overline{R}(\underline{R}(f))\not\leq f$.\\
 \end{itemize}
\item[$(i)\Rightarrow (iii)$] For every $f\in M^X$ and for $x\in X$,
\begin{align*}
\overline{R}\Big(\underline{R}(f)\Big)(x)&=\bigvee_{y\in X}\Big(R(x,y)\odot\bigwedge_{z\in X}\big(R(y,z)\rightarrow f(z)\big)\Big)\\
&\stackrel{P3}{\leq}\bigvee_{y\in X}\Big(R(x,y)\odot\big(R(y,x)\rightarrow f(x)\big)\Big)\\
&=\bigvee_{y\in X}\Big(R(x,y)\odot \big(R(x,y)\rightarrow f(x)\big)\Big), \text{ by the symmetry of R}\\
&\stackrel{P2}{\leq}f(x).
\end{align*}
Thus, $\overline{R}(\underline{R}(f))\leq f$.\\
\item[$(iii)\Rightarrow (i)$] Assume that $R$ is not symmetric, that is there exist $x_0, y_0\in X$ such that $R(x_0, y_0)\neq R(y_0,x_0)$. Consider the following two cases:\\
\begin{itemize}
\item[\textbf{case 1:}] Assume $R(x_0,y_0)\leq R(y_0,x_0)$. Since $R(x_0,y_0)\neq R(y_0,x_0)$, we have $R(y_0,x_0)\not\leq R(x_0,y_0)$ then by \textbf{P9} $R(y_0,x_0)\rightarrow R(x_0,y_0)\lvertneqq\top$.\\
For $f$ defined by $f(x)=\top,$ if $x=y_0$ and $f(x)=\bot$ otherwise. We have,
\begin{align*}
\Big(\underline{R}\big(\overline{R}(f)\big)(y_0)\Big)&=\Bigg(\bigwedge_{x\in X}\bigg(R(y_0,x)\rightarrow\bigvee_{z\in X}\big(R(x,z)\odot f(z)\big)\bigg)\Bigg)\\
&=\Big(\bigwedge_{x\in X}\big(R(y_0,x)\rightarrow R(x,y_0)\big)\Big)\\
&\leq\Big(\big(R(y_0,x_0)\rightarrow R(x_0,y_0)\big)\Big)\lvertneqq \top=f(y_0)\\
\end{align*}
Then, $f\not\leq \underline{R}(\overline{R}(f))$.\\
\item[\textbf{Case 2:}] Let $R(x_0,y_0)\not\leq R(y_0, x_0).$ Then, $R(x_0,y_0)\rightarrow R(y_0, x_0)\lvertneqq \top$, by $\textbf{P9}$. For $f$ defined by $f(x)=\top,$ if $x=x_0$ and $f(x)=\bot$ otherwise. Similarly we have,
\begin{align*}
\Big(\underline{R}\big(\overline{R}(f)\big)(x_0)\Big)&=\Bigg(\bigwedge_{y\in X}\bigg(R(x_0,y)\rightarrow\bigvee_{z\in X}\big(R(y,z)\odot f(z)\big)\bigg)\Bigg)\\
&=\Big(\bigwedge_{y\in X}\big(R(x_0,y)\rightarrow R(y,x_0)\big)\Big)\\
&\leq\Big(\big(R(x_0,y_0)\rightarrow R(y_0,x_0)\big)\Big)\lvertneqq \top=f(x_0)\\
\end{align*}
 Therefore, $f(x_0)\not\leq \underline{R}(\overline{R}(f))(x_0)$. Then,  $f \not\leq \underline{R}(\overline{R}(f))$.
 \end{itemize}
 \end{itemize}
\end{proof}

 Let's look at the case of Euclidean $M$-fuzzy approximation space.
\begin{prop}
 For any $M$-fuzzy approximation space $(X,R)$ the following conditions are equivalent:
\begin{itemize}
\item[(i)] $(X,R)$ is an Euclidean fuzzy approximation space;
\item[(ii)] $\overline{R}(f)\leq\underline{R}(\overline{R}(f))$, for every $f\in M^X$;
    \item[(iii)] $\overline{R}(\underline{R}(f))\leq \underline{R}(f)$, for every $f\in M^X$.
\end{itemize}
\end{prop}
\begin{proof}  Let $R$ be a fuzzy relation on a universe $X$.
\begin{description}
\item[$(i)\Rightarrow (ii)$:] We assume that $R$ is a fuzzy Euclidean relation, i.e. $R(x,y)\odot R(x,z)\leq R(y,z)$, for all $x,y,z\in X$. Then for every $f\in M^X$ and for all every $x\in X$
%\marginpar{Use $\backslash$text\{\} to insert\\ text in math-mode\\ Check refs!}
\begin{align*} 
\underline{R}\big(\overline{R}(f)\big)(x)&=\bigwedge_{y\in X}\Big(R(x,y)\rightarrow \bigvee_{z\in X}\big(R(y,z)\odot f(z)\big)\Big)\\
&\geq \bigwedge_{y\in X}\Big(R(x,y)\rightarrow \bigvee_{z\in X}\big(R(x,y)\odot R(x,z)\odot f(z)\big)\Big), \text{ since R is Euclidean}\\
&=\bigwedge_{y\in X}\Bigg(R(x,y)\rightarrow\bigg( R(x,y)\odot\bigvee_{z\in X}\big( R(x,z)\odot f(z)\big)\bigg)\Bigg), \text{ by Corollary \ref{sec:cor} (item 1)}\\
&\stackrel{P2}{\geq}\bigwedge_{y\in X}\bigvee_{z\in X}(R(x,z)\odot f(z)) = \bigvee_{z\in X}(R(x,z)\odot f(z)) =\overline{R}(f)(x).
\end{align*}
Then, $\overline{R}(f)\leq \underline{R}(\overline{R}(f))$.
\item[$(ii)\Rightarrow (i)$:] We assume that $R$ is not Euclidean. Then there are elements $x_0,y_0,z_0\in X$ such that $R(x_0,y_0)\odot R(x_0,z_0)\not\leq R(y_0,z_0)$. By $\textbf{P9}$ we get $R(x_0,y_0)\odot R(x_0,y_0)\rightarrow R(y_0,z_0)\lneqq \top$. 
Define a map $f:X\to M$ by $f(x)=\top$ if $x=z_0$ and $f(x)=\bot$ otherwise. Then $\displaystyle\overline{R}(f)(y)=\bigvee_{z\in X}(R(y,z)\odot f(z))=R(y,z_0)$, for any $y\in X$. Hence,
\begin{align*}
\overline{R}(f)(x_0)\rightarrow \underline{R}(\overline{R}(f))(x_0) 
&=R(x_0,z_0)\rightarrow \bigwedge_{y\in X}\Big(R(x_0,y)\rightarrow \overline{R}(f)(y)\Big)\\
&=R(x_0,z_0)\rightarrow\bigwedge_{y\in X}\Big(R(x_0,y)\rightarrow R(y,z_0)\Big)\\
&\stackrel{P3}{\leq}R(x_0,z_0)\rightarrow \Big(R(x_0,y_0)\rightarrow R(y_0,z_0)\Big)\\
&\stackrel{P4}{=}\Big(R(x_0,z_0)\odot R(x_0,z_0)\Big)\rightarrow R(y_0,z_0) \lneqq\top.
\end{align*}
Then, again by \textbf{P9}, we have $\overline{R}(f)(x_0)\not\leq\underline{R}(\overline{R}(f))(x_0)$, and  $\overline{R}(f)\not\leq\underline{R}(\overline{R}(f))$.
\item[$(i)\Rightarrow (iii)$:] We assume that $R$ is a fuzzy Euclidean relation.
%i.e. $R(x,y)\odot R(x,z)\leq R(y,z)$, for all $x,y,z\in X$. Then for every $f\in M^X$ and for all every $x\in X$
 For every $f\in M^X$ and $x,y,z\in X$, we have
$R(y,z)\rightarrow f(z)\stackrel{P3}{\leq} \Big(R(x,y)\odot R(x,z)\Big)\rightarrow f(z) = R(x,y)\rightarrow(R(x,z)\rightarrow f(z))$,
%and $\Big(R(x,y)\odot R(x,z)\Big)\rightarrow f(z)\stackrel{P4}{=}R(x,y)\rightarrow(R(x,z)\rightarrow f(z))$.\\
%Hence, for every $f\in M^X$ and every $x,y\in X$,
and 
\begin{align*}
\underline{R}(f)(y)=\bigwedge_{z\in X}\Big(R(y,z)\rightarrow f(z)\Big)
&\leq\bigwedge_{z\in X}\Big(R(x,y)\rightarrow\big(R(x,z)\rightarrow f(z)\big)\Big)\\
&=R(x,y)\rightarrow\bigwedge_{z\in X}\Big(R(x,z)\rightarrow f(z)\Big),\text{ by Corollary \ref{sec:cor} (item 5)}\\
&=R(x,y)\rightarrow \underline{R}(f)(x).
\end{align*}
Thus,
%Therefore, for every $f\in M^X$ and every $x\in X$.
\begin{align*}
\overline{R}(\underline{R}(f))(x) =\bigvee_{y\in X}\Big(R(x,y)\odot\underline{R}(f)(y)\Big)
&\leq \bigvee_{y\in X}\Bigg(R(x,y)\odot \Big(R(x,y)\rightarrow\underline{R}(f)(x)\Big)\Bigg)
\stackrel{P2}{\leq}\underline{R}(f)(x).
\end{align*}
Thus, $\overline{R}(\underline{R}(f))\leq \underline{R}(f)$.
\item[$(iii)\Rightarrow (i)$:] We assume that $R$ is not Euclidean. Then there are elements $x_0,y_0,z_0$ in $X$ such that $R(x_0,y_0)\odot R(x_0,y_0)\rightarrow R(y_0,z_0)\lneqq \top$. 
%
 %Assume that $R$ is not Euclidean, i.e., for some $x_0,y_0,z_0\in X$, $R(x_0,y_0)\odot R(x_0,z_0)\not\leq R(y_0,z_0)$. Then, $\big(R(x_0,y_0)\odot R(x_0,y_0)\big)\rightarrow R(y_0,z_0)$.\\
Consider $f=R_{y_0}$. Then, we have
% \begin{align*}
% &\overline{R}(\underline{R}(f))(x_0)\rightarrow \underline{R}(f)(x_0)\\
% &=\Bigg(\bigvee_{y\in X}\Big(R(x_0,y)\odot\bigwedge_{z\in X}\big(R(y,z)\rightarrow R(y_0,z)\big)\Big)\Bigg)\rightarrow \Bigg(\bigwedge_{z\in X}\Big(R(x_0,z)\rightarrow R(y_0,z)\Big)\Bigg)\\
% &=\bigvee_{y\in X}\Bigg(R(x_0,y)\odot\bigwedge_{z\in X}\Big(R(y,z)\rightarrow R(y_0,z)\Big)\Bigg)\rightarrow\Bigg(\bigwedge_{z\in X}\Big(R(x_0,z)\rightarrow R(y_0,z)\Big)\Bigg)\\
% &\leq\Bigg(R(x_0,y_0)\odot\bigwedge_{z\in X}\Big(R(y_0,z)\rightarrow R(y_0,z)\Big)\Bigg)\rightarrow\Bigg(\bigwedge_{z\in X}\Big(R(x_0,z)\rightarrow R(y_0,z)\Big)\Bigg)\\
% &=\Big(R(x_0,y_0)\odot\top\Big)\rightarrow \Bigg(\bigwedge_{z\in X}\Big(R(x_0,z)\rightarrow R(y_0,z)\Big)\Bigg)\\
% &=R(x_0,y_0)\rightarrow \Bigg(\bigwedge_{z\in X}\Big(R(x_0,z)\rightarrow R(y_0,z)\Big)\Bigg)\\
% &\leq R(x_0,y_0)\rightarrow \Big(R(x_0,z_0)\rightarrow R(y_0,z_0)\Big)\\
% &=\Big(R(x_0,y_0)\odot R(x_0,z_0)\Big)\rightarrow R(y_0,z_0)\\
% &\lneqq\top
% \end{align*}
\begin{align*}
\overline{R}(\underline{R}(f))(x_0)\rightarrow \underline{R}(f)(x_0)
&=\Bigg(\bigvee_{y\in X}\Big(R(x_0,y)\odot\bigwedge_{z\in X}\big(R(y,z)\rightarrow R(y_0,z)\big)\Big)\Bigg)\rightarrow \underline{R}(f)(x_0) \\
&=\bigvee_{y\in X}\Bigg(R(x_0,y)\odot\bigwedge_{z\in X}\Big(R(y,z)\rightarrow R(y_0,z)\Big)\Bigg)\rightarrow\underline{R}(f)(x_0)\\
&\leq\Bigg(R(x_0,y_0)\odot\bigwedge_{z\in X}\Big(R(y_0,z)\rightarrow R(y_0,z)\Big)\Bigg)\rightarrow \underline{R}(f)(x_0)\\
&=\Big(R(x_0,y_0)\odot\top\Big)\rightarrow \underline{R}(f)(x_0)\\
&=R(x_0,y_0)\rightarrow \Bigg(\bigwedge_{z\in X}\Big(R(x_0,z)\rightarrow R(y_0,z)\Big)\Bigg)\\
&\leq R(x_0,y_0)\rightarrow \Big(R(x_0,z_0)\rightarrow R(y_0,z_0)\Big)\\
&=\Big(R(x_0,y_0)\odot R(x_0,z_0)\Big)\rightarrow R(y_0,z_0)
\lneqq\top. %\marginpar{Is one z0 not y0?}
\end{align*}

Thus $\overline{R}(\underline{R}(f))(x_0)\not\leq\underline{R}(f)(x_0)$ and  $\overline{R}(\underline{R}(f))\not\leq\underline{R}(f)$.
\end{description}
\end{proof}
To end this section, let us consider $M$-fuzzy rough sets in approximation space with transitive fuzzy relation.
\begin{prop}
For any $M$-fuzzy approximation space $(X,R)$ the following conditions are equivalent:
\begin{itemize}
\item[(i)] $(X,R)$ is a Transitive fuzzy approximation space;
\item[(ii)] $\underline{R}(f)\leq  \underline{R}(\underline{R}(f))$, for every $f\in M^X$;
\item[(iii)] $\overline{R}(\overline{R}(f))\leq\overline{R}(f)$, for every $f\in M^X$.
\end{itemize}
\end{prop}
\begin{proof} 
We are going two show that $(i)\Leftrightarrow(ii)$ and $(i)\Leftrightarrow(iii)$.
\begin{description}
\item[$(i)\Rightarrow (ii)$:] We assume that $R$ is transitive. Let $f\in M^X$ and  $x\in X$.
\begin{align*}
\underline{R}(\underline{R}(f))(x)&=\bigwedge_{y\in X}\Bigg(R(x,y)\rightarrow \bigwedge_{z\in X}\Big(R(y,z)\rightarrow f(z)\Big)\Bigg)\\
&=\bigwedge_{y\in X}\bigwedge_{z\in X}\Bigg(R(x,y)\rightarrow\Big(R(y,z)\rightarrow f(z)\Big)\Bigg) \text{ by Corollary \ref{sec:cor} (item 5)}\\
&\stackrel{P9}{=}\bigwedge_{y\in X}\bigwedge_{z\in X}(R(x,y)\odot R(y,z))\rightarrow f(z)\\
&\geq \bigwedge_{z\in X}\Big(R(x,z)\rightarrow f(z)\Big) \text{since R is  transitive}\\
&=\underline{R}(f)(x).
\end{align*}
Thus, $\underline{R}(f)\leq \underline{R}(\underline{R}(f))$.
\item[$(ii)\Rightarrow (i)$:]  Assume that $R$ is not transitive. It follows that for some $x_0,y_0\in X$, $\displaystyle\bigvee_{z\in X}(R(x_0,z)\odot R(z,y_0))\not\leq R(x_0,y_0)$.  Then, by \textbf{(P9)}, \Bigg($\displaystyle\bigvee_{z\in X}\Big(R(x_0,z)\odot R(z, y_0)\Big)\Bigg)\rightarrow R(x_0,y_0)\lneqq \top$, for some $x_0,y_0\in X$. We set $f=R_{x_0}$. Note that % we have,
$\displaystyle\underline{R}(f)(x_0)=\bigwedge_{z\in X}(R(x_0,z)\rightarrow R(x_0,z))=\top$.
Therefore, 
\begin{align*}
\underline{R}(f)(x_0)\rightarrow\underline{R}(\underline{R}(f))(x_0)
&=\top\rightarrow\Bigg(\bigwedge_{z\in X}\Big(R(x_0,z)\rightarrow\bigwedge_{y\in X}\big(R(z,y)\rightarrow R(x_0,y)\big)\Big)\Bigg)\\
&=\bigwedge_{z\in X}\Big(R(x_0,z)\rightarrow \bigwedge_{y\in X}\big(R(z,y)\rightarrow R(x_0,y)\big)\Big)\\
&=\bigwedge_{z\in X}\bigwedge_{y\in X}\Big(R(x_0,z)\rightarrow \big(R(z,y)\rightarrow R(x_0,y)\big)\Big)\\
&=\bigwedge_{z\in X}\bigwedge_{y\in X}\Big(\big(R(x_0,z)\odot R(z,y)\big)\rightarrow R(x_0,y)\Big)\\
&\leq\bigwedge_{z\in X}\Bigg(\bigvee_{y\in X}\Big(R(x_0,z)\odot R(z,y)\Big)\rightarrow R(x_0,y)\Bigg)\\
&\leq \Bigg(\bigvee_{y\in X}\Big(R(x_0,z)\odot R(z,y_0)\Big)\Bigg)\rightarrow R(x_0,y_0)\lneqq\top,
\end{align*}
Thus $\underline{R}(f)(x_0)\rightarrow \underline{R}(\underline{R}(f))(x_0)\lneqq \top$, i.e. $\underline{R}(f)(x_0)\not\leq\underline{R}(\underline{R}(f))(x_0)$, and $\underline{R}(f)\not\leq\underline{R}(\underline{R}(f))$.
\item[$(i)\Rightarrow (iii)$:] Let $f\in M^X$ and $x\in X$.
\begin{align*}
\overline{R}(\overline{R}(f))(x)&=\bigvee_{y\in X}(R(x,y)\odot \bigvee_{z\in X}(R(y,z)\odot f(z)))\\
&=\bigvee_{y\in X}\bigvee_{z\in X}(R(x,y)\odot R(y,z)\odot f(z)),\text{ by Corollary \ref{sec:cor} (item 1)}\\
&\stackrel{P3}{\leq}\bigvee_{y\in X}\bigvee_{z\in X}(R(x,z)\odot f(z)) \textup{ by transitivity of } R\\
&=\overline{R}(f)(x).
\end{align*}
Thus, $\overline{R}(\overline{R}(f))\leq \overline{R}(f)$.
\item[$(iii)\Rightarrow (i)$:]
We assume that $R$ is not transitive. Then
 $\displaystyle\bigvee_{y\in X}(R(x_0,y)\odot R(y,z_0))\not \leq R(x_0,z_0)$ for some $x_0,z_0\in X$. By \textbf{(P9)}, $\displaystyle \bigvee_{y\in X}(R(x_0,y)\odot R(y,z_0))\rightarrow R(x_0,z_0)\lneqq\top$.
Define a map $f$ by $f(z_0)=\top$, and $f(x)=\bot$ if $x\neq z_0$. We have $\displaystyle\overline{R}(f)(x_0)=\bigvee_{z\in X}\Big(R(x_0,z)\odot f(z)\Big)=R(x_0,z_0)$.
Therefore,
\begin{align*}
\overline{R}(\overline{R}(f))(x_0)\rightarrow \overline{R}(f)(x_0)
&=\Bigg(\bigvee_{y\in X}\Big(R(x_0,y)\odot\bigvee_{z\in X}\big(R(y,z)\odot f(z)\big)\Big)\Bigg)\rightarrow R(x_0,z_0)\\
&=\Bigg(\bigvee_{y\in X}\bigvee_{z\in X}\Big(R(x_0,y)\odot \big(R(y,z)\odot f(z)\big)\Big)\Bigg)\rightarrow R(x_0,z_0)\\
&=\Bigg(\bigvee_{y\in X}\Big(\bigvee_{z\in X}\big(R(x_0,y)\odot R(y,z)\big)\odot f(z)\Big)\Bigg)\rightarrow R(x_0,z_0)\\
&=\Bigg(\bigvee_{y\in X}\Big(R(x_0,y)\odot R(y,z_0)\Big)\Bigg)\rightarrow R(x_0,z_0)
\ \lneqq\ \top,
\end{align*}
and $\overline{R}(\overline{R}(f))(x_0)\not\leq\overline{R}(f)(x_0)$. Thus, $\overline{R}(\overline{R}(f))\not\leq\overline{R}(f)$.
\end{description}
\end{proof}
\begin{theorem} Let $(X, R)$  be an $M$-fuzzy approximation space, and  $f\in M^X$ be a fuzzy subset. 
\begin{itemize}
\item[(i)] If $R$ is  an $M$-tolerance relation then 
\[
\underline{R}(\underline{R}(f))=\underline{R}(f) \le f \le \overline{R}(f) = \overline{R}(\overline{R}(f)).\]
\item[(ii)] If $R$ is an $M$-equivalence relation then
\[\overline{R}(\underline{R}(f)) = \underline{R}(\underline{R}(f)) = \underline{R}(f) \le f \le \overline{R}(f)  = \overline{R}(\overline{R}(f)) =  \underline{R}(\overline{R}(f)).\]
\end{itemize}
\end{theorem}

% \begin{cor}
% For any $M$-fuzzy approximation space $(X, R)$ the following conditions hold,
% \begin{itemize}
% \item[(i)] If $R$ is  reflexive and transitive then
% \[
% \overline{R}(\overline{R}(f))=\overline{R}(f) \quad \mbox{ and }\quad
% \underline{R}(\underline{R}(f))=\underline{R}(f)\quad  \text{ for every } f\in M^X.\]
% \item[(ii)] If $R$ is an $M$-equivalence relation then
% \[\overline{R}(\underline{R}(f))=\underline{R}(f)\quad \text{ and }\quad  \underline{R}(\overline{R}(f))=\overline{R}(f), \quad \text{ for every } f\in M^X.\]
% \end{itemize}
% \end{cor}

% \begin{itemize}
% \item[-] $\overline{R}(\underline{R}(f))=\underline{R}(f)$, for every $f\in M^X$
%     \item[-] $\underline{R}(\overline{R}(f))=\overline{R}(f)$, for every $f\in M^X$
% \end{itemize}

\begin{proof}
\begin{itemize}
\item[$(i)$]
For every $f\in M^X$. We have by Transitivity of $R$, $\overline{R}(\overline{R}(f))\leq\overline{R}(f)$. Since $R$ is reflexive, we have, $\overline{R}(f)\leq \overline{R}(\overline{R}(f))$. Hence $\overline{R}(\overline{R}(f))=\overline{R}(f)$.
The second equality can be proved in the analogous way.

\item[$(ii)$] For every $f\in M^X$, we have by the transitivity of $R$, $\underline{R}(f)\leq\underline{R}(\underline{R}(f))$.
Next, by the reflexivity of $R$ we have $\underline{R}(f)\leq\overline{R}(\underline{R}(f))$.
Then, $\overline{R}(\underline{R}(f))\leq\overline{R}(\underline{R}(\underline{R}(f))).$
By the symmetry of $R$ we have $\overline{R}(\underline{R}(\underline{R}(f)))\leq\underline{R}(f)$. Then $\overline{R}(\underline{R}(f))\leq\underline{R}(f)$. 
Therefore, $\underline{R}(f)=\overline{R}(\underline{R}(f))$.
By the same manner we can prove the second equality.
\end{itemize}
\end{proof}
\section{Conclusion}\label{sec:conclusion}
In this paper, we extended Anna Maria Radzikowska work's \cite{RK2004} by generalizing the notion of fuzzy rough sets with an arbitrary residuated multilattice as the set of truth values.
We defined the fuzzy approximation space and prove that the defining lower and upper $M$-fuzzy rough approximators coincide with the lower and upper $L$-fuzzy rough approximators defined by Anna Maria Radzikowska \cite{RK2004}.
Since we established the fuzzy rough set with a residuated multilattice as underlying set of truth values, it will be also interesting to use the residuated multilattice as underlying set of truth values in Rough Fuzzy Concept Analysis.

% \paragraph{Additional Questions/Requests}
% \begin{enumerate}
%     \item Get example of residuated multillattices and use it/these to provide examples of M-fuzzy rough sets of each classes.
%     \item Is it known what structure has the image of residuated M-lattices?
%     \item Has the ideal/filter completion of multilattices been described? If not then we should do it right now.
% \end{enumerate}


\addcontentsline{toc}{chapter}{BIBLIOGRAPHIE}
\begin{thebibliography}{99}

\bibitem{BK2014}{E. Bartl, Jan Konecny, \emph{Formal L-concepts with Rough Intents}, In CLA 2014: proceeding of the 11th International Conference on Concept lattice and their Applications, (2014), 207-218.}
\bibitem{BK2015}{E. Bartl, Jan Konecny, \emph{Using Linguistic Hedges in L-rough Concept Analysis}, CLA 2015, ISBN 978-2-9544948-0-7, (2015), 229-240.}

\bibitem{BK2017}{E. Bartl, Jan Konecny, \emph{Rough Fuzzy Concept Analysis}, Fundamenta Informaticae 15, (2017), 141-168. DOI 10.3233/FI-2017-1601 IOS Press}


\bibitem{RV2005}{R. Belohlavek, V. Vychodil, \emph{What is a Fuzzy Concept Lattice?}, In R. Belohlavek et al.  eds. CLA, (2005), 34-45.}

\bibitem{BR2001}{R. Belohlavek, \emph{Fuzzy Closure Operators}, J. Math. Anal. Appl. 262(2) (October 2001) 473-489.}

\bibitem{BM1955}{M. Benado, \emph{Les ensembles partiellement ordonn\'es et th\'eor\`eme de raffinement de Schreier II (Th\'eorie des multistructures)}, Czechoslovak Math. J.5, (1955), 308 - 344.}

\bibitem{BF1994}{ A. Burusco, R. Fuentes-Gonzalez, \emph{The study of L-fuzzy concept lattice}, Mathware and Soft Computing 3, (1994), 209-218.}

\bibitem{CCGMO2011}{I. P. Cabrera, P. Cordero, G. Guti\'errez, J. Martinez, M. Ojeda-Aciego, \emph{Residuated operations in hyperstructures: residuated multilattices}, 11th International Conference on Computational and Mathematical Methods in Scinece and Engineering, CMMSE 2011, (June 2011), 26-30.}

\bibitem{CCGMO2014}{I. P. Cabrera, P. Cordero, G. Guti\'errez, J. Martinez, M. Ojeda-Aciego, \emph{On residuation in multilattices: Filters, congruences, and homomorphisms}, Fuzzy Sets and Systems, (2014), 1-21.}

\bibitem{DP2002}{ B. Davey, H. Priestley, \emph{Introduction to Lattices and Order}, Cambridge University Press, second edition, 2002.}

\bibitem{DW1939}{R. P. Dilworth, N. Ward, \emph{Residuated lattices.}, Trans. Amer. Math. Soc. 45, (1939), 335-354.}

\bibitem{DP1990}{D. Dubois, H. Prade, \emph{Rough fuzzy sets and fuzzy rough sets}, Int. J. of General Systems, 17, (1990), 191 - 209.}

\bibitem{GW1999}{B. Ganter, R. Wille, \emph{Formal Concept Analysis: Mathematical Foundations}, Springer, Berlin- Heindelberg (1999).}
\bibitem{GO1967}{J.A Goguen, \emph{L-fuzzy sets}, Journal of Mathematical Analysis and Applications 18, (1967), 145 - 174.}

\bibitem{MKLK2015}{L. N. Maffeu Nzoda, B.B.N. Koguep, C. Lele, L. Kwuida, \emph{Fuzzy setting of residuated multilattices}, Annals of Fuzzy Mathematics and Informatics, volume 10, No 6, (December 2015),  929 - 948}

\bibitem{MOR2013}{J. Medina, M. Ojeda-Aciego, J. Ruiz-Calvino, \emph{Concept-forming operators on multilattices}, 11th International Conference (ICFCA 2013), Dresden, Germany, May 21- 24, Volume 7880, (2013), 203 - 215.}

\bibitem{MR2009}{J. Medina, J. Ruiz-Calvi\~no, \emph{Fuzzy formal concept analysis via multilattices: first prospects and results}, In the 9th International Conference on Concept Lattices and theirs Applications (CLA 2012), 69 - 79.}

\bibitem{MOR2007}{ J. Medina, M. Ojeda-Aciego, J. Ruiz-Calvi\~no, \emph{Fuzzy logic programming via multilattices}, Fuzzy Sets Syst, 158 (6), (2007), 674 - 688.}

\bibitem{PIKD2014}{ J. Poelmans, Dmitry I. Ignatov, Sergei O. Kuznetsov, Guido Dedene, \emph{Fuzzy and Rough Formal Concept Analysis: a Survey}, International Journal of General Systems, Vol. 43, No. 2, (2014), 105 - 134.}

\bibitem{RK2004}{A. M. Radzikowska, E. E. Kerre, \emph{Fuzzy rough sets based on residuated lattices}, Transactions on Rough Sets II, in: LNCS, vol. 3135, Springer-Verlag, (2004), 278 - 296.}

\bibitem{ZL1965}{ L. A. Zadeh. \emph{Fuzzy sets}, Information and Control 8, (1965), 338 - 358.}

\bibitem{PZ1982}{ Zdislaw Pawlak, \emph{Rough sets}, International Journal of Computer and Information Sciences, 11(5), ( 1982), 341 - 356. }
\bibliographystyle{plain}
\end{thebibliography}
\end{document}